\documentclass[10pt]{article}

\usepackage{ a4wide,amsmath,amsthm,amssymb}

\newtheorem{problem}{Problem}

\newtheorem{theorem}{Theorem}[section]

\newtheorem{lemma}[theorem]{Lemma}

\newtheorem{prop}[theorem]{Proposition}

\newcommand{\arr}{\rightarrow}

\newcommand{\tp}{\mathop{\mathsf{rist}}\nolimits}
\newcommand{\rs}{\mathop{\mathsf{rist}}\nolimits}
\newcommand{\Rs}{\mathop{\mathsf{Rist}}\nolimits}
\newcommand{\st}{\mathop{\mathrm{Stab}}\nolimits}
\newcommand{\at}{\mathop{\mathrm{Aut}}\nolimits}
\newcommand{\pol}{\mathop{\mathrm{Pol}}\nolimits}
\newcommand{\fat}{\mathop{\mathrm{FAut}}\nolimits}
\newcommand{\rat}{\mathop{\mathrm{RAut}}\nolimits}

\newcommand{\ii}{\mathop{\mathcal{I}}\nolimits}

\newcommand{\x}{\mathop{\mathsf{X}}\nolimits}

\newcommand{\lc}{\langle}
\newcommand{\rc}{\rangle}

\author{Yaroslav Lavrenyuk\\\small{\it Faculty of Mechanics and Mathematics,
Kyiv Taras Shevchenko University, Kyiv, Ukraine}\\
\small{\tt ylavrenyuk\symbol{64}univ.kiev.ua}}

\title{Minimal generating sets in wreath products}

\begin{document}

\maketitle

\abstract{We find some sufficient conditions under which the permutational
wreath product of two groups  has a minimal generating set. In
particular we prove that for a regular rooted tree
 the group of all automorphisms and the group
of all finite-state automorphisms of such a tree satisfy these
conditions.}

\section{Introduction}
We consider the following problem
\begin{problem}
\label{prob_Dom}
Let $T$ be a regular rooted tree. Do the group of all
automorphisms and the group of all finite-state automorphisms of
$T$ have any minimal generating set?
\end{problem}

Originally such a problem was posted by B. Cs\'{a}k\'{a}ny and F.
G\'{e}cseg \cite{csac_gecen} in terms of automata in 1965. They
asked whether or not the semigroup of all automaton
transformations, the semigroup of all finite automaton
transformations, the group of all bijective automaton
transformations, and the group of all finite bijective automaton
transformations over a fixed finite alphabet with at least two
elements have a minimal generating set?

Answer for semigroups is negative. This result was obtained
independently by S. Aleshin \cite{Aleshin1} in 1970 and P.
D\"{o}m\"{o}si \cite{Domosi1} in 1972. The question about groups
(i.e. Problem~\ref{prob_Dom}) was also formulated in papers of P.
D\"{o}m\"{o}si. Particularly it was in \cite[Problem
2.1]{Domosi_probl} and \cite[Problem 2.31]{Domosi_Nehaniv}.
Moreover this problem is mentioned in the
papers~\cite{aleshin:automata, problaut}.

Among works related to this problem we mention the result of
Andriy Oliynyk from \cite{aol11}. Namely, it was proven that
finite-state wreath product of transformation semigroups is not
finitely generated and in some cases doesn't have a minimal
generating set. We also mention papers devoted to the study of
generating sets in projective limits of wreath products of groups
\cite{bhattacharjee,quick,bondarenko,lucchini}.

We find some sufficient conditions under which the permutational
wreath product of finite group and infinite group  has a minimal
generating set (Theorem~\ref{tm_2011_100}). We also give a few
examples of groups and classes of groups satisfying such
conditions. In particular we prove that for a regular rooted tree
 the group of all automorphisms and the group
of all finite-state automorphisms of such a tree satisfy these
conditions (Theorems~\ref{cor_2011_121} and \ref{cor_2011_121_2}).
Therefore these groups have minimal generating sets. Thus
Problem~\ref{prob_Dom} is solved positively.

Most results of this paper were announced without proofs in
~\cite{lavr:mgsen,lavr:mgs2en}. Application of
Theorem~\ref{tm_2011_100} to wider class of wreath branch
groups will be the subject of forthcoming paper.

 The author gratefully acknowledges the many helpful suggestions of Ievgen Bondarenko, Andriy Oliynyk, and Wital
Sushchansky during the preparation of the paper.

\section{Main theorem}
We first recall the notion of the permutational wreath product.

Let $(A,X)$ be a permutation group and let $H$ be a group. Then
\emph{the permutational wreath product} $(A,X)\wr H$ is the
semi-direct product $(A,X)\rightthreetimes H^X$, where $(A,X)$
acts on the direct power $H^X$ by the respective permutations of
the direct factors.

We will say that a permutation group $(A,X)$ satisfies the
condition {\bf PS} if:
\begin{enumerate}
\item
The group $(A,X)$ is finite and transitive.
\item
There are subsets $X_1,X_2$ of $X$ and subgroups $A_1,A_2$ of $A$
with the following properties:
\begin{itemize}
\item
$(A_1,X)$ and $(A_2,X)$ act transitively on $X_1$ and $X_2$
respectively and act trivially on $X\setminus X_1$ and $X\setminus
X_2$ respectively.
\item
$X_1$ and $X_2$ do not intersect.
\item
$|X_1|\ge 2$, $|X_2|\ge 3$.
\item
If $|X_1|= 2$ then there is  $a\in A$ satisfying $a(X_1)\cap
X_2\ne
\emptyset$ and $a(X_1)\nsubseteq X_2$.
\end{itemize}
\end{enumerate}

We say that a group $G$ satisfies \emph{the L-condition}, if $G$
is decomposed into permutational wreath product $G= (A,X)\wr H$
and there are a normal subgroup $H_0$ of $H$ and an integer $k>1$
with the following properties:
\begin{enumerate}
\item
The permutation group $(A,X)$ satisfies the condition {\bf PS},
\item
The quotient $H/H_0$ has infinite minimal generating set,
\item
$|H/H_0|\ge|H_0|$,
\item
Either $H_0<H'$ or $H'\lneq H_0<H^kH'$, where $H'$ is the
commutator subgroup of $H$ and $H^k=\langle
\{h^k , h\in H\}\rangle$,
\item
If $H'\lneq H_0<H^kH'$ then there is a subset $C$ of some minimal
generating set of $H/H_0$ such that
\begin{enumerate}
\item
$|C|=|H/H_0|$,
\item
For every coset $c\in C$ there is $h\in c$ of order $k$.
\end{enumerate}
\end{enumerate}

Note that condition 2 of the definition of the L-condition imply
that  $H$ is infinite groups.

The main result of the paper is the following theorem.

\begin{theorem}
\label{tm_2011_100}
A group with the L-condition has a minimal generating set.
\end{theorem}

\subsection*{Proof of Theorem \ref{tm_2011_100}}

At first we fix some notations.

Let $G= (A,X)\wr H$ and let $H_0$ be a normal subgroup of $H$. We
assume that $G,$ $(A,X)$, $H$, and $H_0$ satisfy the conditions of
Theorem~\ref{tm_2011_100}. Let $X=\{ {0}, {1},\ldots,
 {n}\}$, $X_1=\{ {0}, {1},\ldots,  {l_1}\}$, and
$X_2=\{ {l_2}, {l_2+1},\ldots,  {n}\}$.

The group $G$ is a semidirect product of its subgroups $A$ and
$K$, where  $K$ is the direct product of $n+1$ copies of $H$,
i.e., $K=\underbrace{H\times \ldots \times H}_{n+1}$. We will also
write whole subgroup $K$ as $(\underbrace{H,\ldots,H}_{n+1})$. The
conjugation of $g=(g_0,\ldots,g_{n})$ by an element of $(A,X)$ is
a corresponding permutation of coordinates of the tuple.

Without loss of generality, we will make the following
assumptions:

 If there exists $a\in A$
such that $a(X_1)\cap X_2\ne
\emptyset$ and $a(X_1)\nsubseteq X_2$ then let $d_1\in A$ be such that $d_1(
{0})\notin X_2$ and $d_1( {1})= {n}$.

Otherwise, if $a(X_1)\cap X_2\ne
\emptyset$ implies that $a(X_1)\subseteq X_2$ for all $a\in A$
then $|X_1|\ge 3$ by the condition {\bf PS}.  In this case let
$d_2\in A$ be such that $d_2(0)=n$, $d_2(1)=n-1$, and
$d_2(2)=n-2$.

By the L-condition there is a minimal generating set
$\bar{F}=\{\bar{f}_i\ |\ i
\in {\ii}\}$ of $H/H_0$, where $\ii$ denotes a set of indices.
Let $\psi : H\arr H/H_0$ be the canonical epimorphism.
 For every $i\in \ii$ we fix some element $f_i\in
H$ such that $\psi(f_i)=\bar{f_i}$. Denote
$$F=\{f_i
\ |   i\in \ii\}.$$

Let $\ii_1$ and $\ii_2$ be subsets of $\ii$ with the following
properties:
\begin{itemize}
\item
$|\ii_2|=|\ii|.$
\item
$\ii_1=\ii\setminus \ii_2$.
\end{itemize}

Denote $$F^{\ii_j}=\{f_i \ |   i\in \ii_j\}\ \text{for} \ j=1,2.$$

In the case of $H'\lneq H_0<H^kH'$ due to condition 5 of the
L-condition we can assume that for every $i
\in
\ii_2$ the following equality holds: $f_i^k=e$. In the case of
$H_0<H'$ we can assume that $\ii_2=\ii$.

Since $\bar{F}$ is an infinite set  the set of the finite words
over $\bar{F}$ has the same cardinality as $\bar{F}$. By the
L-condition $|H/H_0|\ge |H_0|.$ It follows that
$$|\ii_2|=|\ii|=|H/H_0|\ge |H_0|.$$

Therefore we can fix a surjection $\phi: \ii_2
\arr H_0$. We also define the set of elements of $G$:
$$S_K=\{q_{i}=(f_i,e,\ldots,e,\phi(i))
\ | \ i\in \ii_2\}\cup \{q_{i}=(f_i,e,\ldots,e)
\ | \ i\in \ii_1\}.$$

Let us fix a minimal generating set of $A$:
$S_A=\{s_1,s_2,\ldots,s_r\}$. Let $S =S_K \cup S_A$ and $N=\langle
S
\rangle$.
Note that $A=\langle S_A \rangle $ is contained in $N$.

\begin{lemma}
\label{l014_211}
The set $S$ is a minimal generating set of the group $N$.
\end{lemma}
\begin{proof}
Since $G$ is the semidirect product $A\rightthreetimes K$ and
$S_K\subset K$, any element $s$ of $S_A$ cannot be written as an
product of  elements of $S\setminus \{s\}$ and their inverses.

Further, suppose, contrary to our claim, that the element $q_i$
for some $i\in \ii$ is a product of elements of $S\setminus
\{q_{i}\}$ and their inverses. It is easy to check that this decomposition of
$q_i$ can be transformed to the product of the form
$q_i=(q_{i_1}^{\epsilon_1})^{a_{1}}\ldots
(q_{i_m}^{\epsilon_m})^{a_{m}}$, where $i_1,\ldots, i_m \in
\ii\setminus\{i\}$, $\epsilon_1,\ldots,\epsilon_m\in \{-1,1\}$ and $a_1,\ldots,a_m\in A$. Consider
 $0$-th coordinate of $q_i$. We have that
$f_i$ is a product of elements of $F\setminus
\{f_i\}$, their inverses and elements of $H_0$.
Applying $\psi$ we conclude that $\bar{f}_i$ is a product of
elements of $\bar{F}\setminus \{\bar{f}_i\}$ and their inverses.
This contradicts to the fact that  $\bar{F}$ is a minimal
generating set of the quotient $H/H_0$.
\end{proof}

We next show that the set $S$ is a generating set of $G$, i.e., we
next show that $N=G$.

\begin{lemma}
\label{l0165_11}
For every $g\in \lc F\rc$ the elements
$u_{n-2}=(e,\ldots,e,{g},e,g^{-1})$ and
$u_{n-1}=(e,\ldots,e,{g},g^{-1})$ are contained in $N$.
\end{lemma}
\begin{proof}
The element $g$ can be decomposed into the product of elements of
$F$ and their inverses: $g=f_{i_1}^{\epsilon_1}\ldots
f_{i_m}^{\epsilon_m}$, where $\epsilon_1,\ldots,\epsilon_m\in
\{-1,1\}$. For every $j\in X_1
\setminus\{0\}$ choose  $b_j\in A_1$ such that $b_j( {0})= {j}$. Then
$$t_{j}=b_j^{-1}q_{i_1}^{\epsilon_1}\ldots q_{i_m}^{\epsilon_m}b_j (q_{i_1}^{\epsilon_1}\ldots
q_{i_m}^{\epsilon_m})^{-1}=(g^{-1},e,\ldots,e,g,e,\ldots,e)\in
N,$$ where $g$ is located on $j$-th coordinate of the tuple.

We consider all possible cases depending on the group $A$. We will
use here the elements $d_1$ and $d_2$ which were defined at the
beginning of the proof of the theorem.
\begin{enumerate}
\item
There is  $a\in A$ such that $a(X_1)\cap X_2\ne \emptyset$ and
$a(X_1)\nsubseteq X_2$. For every $m \in \{n-2,n-1\}$ choose
$c_m\in A_2$ such that $c_m(n)=m$. Then
$t_{1}^{d_1}(t_{1}^{d_1c_m})^{-1}=u_m\in N$ for $m=n-2, n-1$.
\item
For all $a\in A$, the inequality $a(X_1)\cap X_2\ne \emptyset$
implies that $a(X_1)\subseteq X_2$. Then $|X_1|\ge 3$ and elements
$u_{n-1}=t_1^{d_2}$ and $u_{n-2}=t_2^{d_2}$ are contained in $N$.
\end{enumerate}
\end{proof}

\begin{lemma}
\label{l018_1111}
The subgroup $(E,\ldots,E,H')$ of $K$ is contained in $N$.
\end{lemma}
\begin{proof}
Let $h_1,h_2\in H$. Then there exist $g_j\in
\lc F\rc$,  $i_j \in \ii_2$ for $j=1,2$ such that $h_j=g_j\phi(i_j)$.
By the construction, the set $S$ contains elements
 $q_{i_j}=(f_{i_j},e,\ldots,e,\phi(i_j))$ for $j=1,2$.
Let $a\in A_1$ be such that $a(0)=1$. Then
$q_{i_2}^a=(e,f_{i_2},e,\ldots,e,\phi(i_2))\in N$. By
Lemma~\ref{l0165_11} elements $t_1=(e,\ldots,e , g_1^{-1} ,e,g_1)$
and $t_2=(e,\ldots,e , g_2^{-1} ,g_2)$ are contained in $N$.
Therefore $h_1'=t_1q_{i_1}=(f_{i_1},e,\ldots,e , g_1^{-1}
,e,g_1\phi(i_1))$ and $h_2'=t_2q_{i_2}^a=(e,f_{i_2},e,\ldots,e ,
g_2^{-1} ,g_2\phi(i_2))$ are contained in $N$. Hence
$h_1'h_2'{h_1'}^{-1}{h_2'}^{-1}=(e,\ldots,e ,
h_1h_2{h_1}^{-1}{h_2}^{-1})\in N$. Thus $(E,\ldots,E,H')<N$.
\end{proof}

\begin{lemma}
\label{l018_112}
If $H'\lneq H_0<H^kH'$ then $(E,\ldots,E,H^k)<N$.
\end{lemma}
\begin{proof}
If $H'\lneq H_0<H^kH'$ then for every $i
\in
\ii_2$ the following equality holds: $f_i^k=e$.

Let $h
\in H$. Then $h=gh_0$ for some $g\in \lc F\rc$ and $h_0\in H_0$.
Since $H_0>H'$ and $F^{\ii_1}\cup F^{\ii_2} = F$ there exist
$g_1\in \lc F^{\ii_1}\rc$, $g_2\in
\lc F^{\ii_2}\rc$, and $h_1\in H'$ such that $g=g_1g_2h_1$.
Since $H_0>H'$ there exists $i \in \ii_2$ satisfying
$\phi(i)=h_1h_0$. Thus we have $h=g_1g_2\phi(i)$. By the
construction, the set $S$ contains the element
$q_{i}=(f_{i},e,\ldots,e,\phi(i))$.  By Lemma~\ref{l0165_11}
element $t_2=(e,\ldots,e , g_2^{-1} ,g_2)$ is contained in $N$.
Let $a\in A$ be such that $a(0)=n$. Note that element
$t_1=a^{-1}(g_1,e,\ldots,e)a=(e,\ldots,e,g_1)$ is contained in
$N$. Therefore $h'=t_1t_2q_i =(f_{i},e,\ldots,e ,
g_2^{-1},g_1g_2\phi(i))$ is contained in $N$. By the condition of
the lemma there is $h_2\in H'$ such that $g_2^p=h_2$. Let $a_1\in
A_2$ be such that $a_1(n-1)=n$. Then $(e,\ldots,e,h_2,e)=
a_1^{-1}(e,\ldots,e,h_2)a_1\in N$ by Lemma~\ref{l018_1111}.
Therefore ${h'}^k(e,\ldots,e,h_2,e)=(e,\ldots,e , e, h^k)\in N$.
Thus $(E,\ldots,E,H^k)<N$.
\end{proof}

\begin{lemma}
\label{l02111}
$N=G$.
\end{lemma}
\begin{proof}
If $H_0<H'$ then $(E,\ldots,E,H_0)< N$ by Lemma~\ref{l018_1111}.
If $H_0>H'$ then the conditions of Lemma~\ref{l018_112} hold by
condition 5 of the L-condition. Thus in this case
$(E,\ldots,E,H_0)< N$ too. By construction $(F,E,\ldots, E)$ is
contained in $\lc S_K, (E,\ldots,E,H_0) \rc$. Therefore the set
$(F,E,\ldots, E)$ is contained in $N$. Let $a\in A$ be such that
$a(0)=n$. Then $a^{-1}(F,E,\ldots, E)a=(E,\ldots, E,F)\subset N$.
Since $\lc F\rc H_0=H $ we have $(E,\ldots, E,H)<N$. Also by
transitivity of $(A,X)$ we get $(H,\ldots, H)<N$. It follows that
$N=G$.
\end{proof}

Now the assertion of Theorem~\ref{tm_2011_100}  follows
immediately from Lemma~\ref{l014_211} and Lemma~\ref{l02111}.

\section{Applications and examples}

We first  give natural  constructions of groups with property {\bf
PS}.

\begin{prop}
\label{pr_1}
The following groups satisfy {\bf PS}:
\begin{enumerate}
\item
The symmetric group of degree $m\ge 5$.
\item
The permutational wreath product $(B_1,Y_1)\wr (B_2,Y_2)$, where
$(B_1,Y_1)$ and $ (B_2,Y_2)$ are finite transitive permutation
groups and $|Y_1|\ge 2$, $|Y_2|\ge 3$.
\item
The permutational wreath product  $(B_1,Y_1)\wr (B_2,Y_2)\wr
(B_3,Y_3)$, where $(B_i,Y_i)$ is finite transitive permutation
group and $|Y_i|\ge 2$ for $i=1,2,3$.
\end{enumerate}
\end{prop}

Now we formulate two corollaries from Theorem~\ref{tm_2011_100}
which are more applicable.

\begin{prop}
\label{tm_2011_1121}
Let $G= (A,X)\wr H$ and there is an integer $k>1$ with the
following properties:
\begin{enumerate}
\item
$(A,X)$ satisfies {\bf PS}.
\item
$H$ is an infinite group.
\item
$H'>H^k$.
\item
$|H'|\le |H/H'|$.
\end{enumerate}
Then the group $G$ satisfies the L-condition.
\end{prop}
\begin{proof}
Due to Theorem~\ref{tm_2011_100} we only need to show that  $H/H'$
has infinite minimal generating set. By the conditions of the
theorem  $H/H'$ has exponent $p$. We conclude from results of
\cite[Theorem 3.1 and Lemma 4.3]{ruzicka} that an abelian group of
bounded exponent has minimal generating set, and the proposition
follows.
\end{proof}

\begin{prop}
\label{tm_2011_112}
Let $G= (A,X)\wr H$ and there is an infinite subgroup $M$ of $H$
with the following properties:
\begin{enumerate}
\item
$(A,X)$ satisfies {\bf PS}.
\item
$M$ has exponent 2.
\item
$|M|=|H|$.
\item
$M\cap H^2 = E$.
\end{enumerate}
Then the group $G$ satisfies the L-condition.
\end{prop}
\begin{proof}
Put $H_0=H^2$. Then $H/H_0$ and $M$ have minimal generating sets
(Hamel basis) as vector spaces over the field with two elements.

Let $I$ be an index set, and let $B=\{b_i\ |
\ i\in I\}$ be a minimal generating set of $M$. Let also $\psi:H\arr
H/H^2$ be the canonical epimorphism. Since $M\cap H^2 = E$ the
restriction $\psi$ onto $M$ is a bijection and the set $\psi(B)$
is a minimal generating set of $\lc
\psi(B)
\rc$.
Since $B$ is infinite we have $|\psi(B)|=|B|=|M|=|H|$. Therefore
we have $|H|=|H/H_0|$ and $\psi(B)$ can be complemented to Hamel
basis $F$ of the space $H/H_0$.
 It is also
evident that $b_i^2=e$ for every $b_i\in B$. Note that inclusion
$H'<H^2$ is always true. Thus the group $G$ satisfies the
L-condition.
\end{proof}

Remark that we use existence of Hamel basis of a vector space over
a field in the proof of Proposition~\ref{tm_2011_112}. Hence this
proof uses the axiom of choise in some cases.

In the next section we will apply Proposition~\ref{tm_2011_112} to
some groups of automorphisms of rooted trees, and particularly
give positive answer to Problem~\ref{prob_Dom}.

\subsection*{Automorphism groups of rooted trees}
We first recall necessary definition related to rooted trees and
groups acting on rooted trees. All notions which will be defined
in this section are well-known, see for
instance~\cite{grineksu_en,sidki_monogr,sid:cycl} for more
details.

Let us fix our notations. Let $\x=(\x_1,\x_2,\ldots)$ be a
sequence of finite sets $\x_i=\{ {0}, {1},\ldots, {n_i}\}$ (we
assume $n_i\ge 1$ for all $i$). Let $\x^n$ denote the set of all
words of the form $x_1x_2\ldots x_n$, where $x_i\in
\x_i$ for $i=1,\ldots,n$. Let $\x^*$ denote the set which consist of the empty word $\emptyset$ and all words of the form
 $x_1x_2\ldots x_n$, where $n\in \mathbb{N}$ and $x_i\in
\x_i$ for $i=1,\ldots,n$. Let $\x^\omega$ denote the set
of all infinite words of the form $x_1x_2\ldots$, where $x_i\in
\x_i$. We denote by ${\x}^{(k)}$ the infinite sequence $({\x}_k,{\x}_{k+1},\ldots)$.

We can consider the set of words  $\x^*$ as rooted tree $T_{\x}$
which can be defined as follows: a vertex $x_1x_2\ldots x_n$ is
adjacent to  $x_1x_2\ldots x_{n-1}$, $\emptyset$ is the root. For
rooted tree $T_{\x}$ we also define \emph{vertex subtree} $T_v$
($v\in
\x^*$) which vertices are the words of the form $v\x^*$. We call the set of vertices $\x^n$
\emph{the $n$-th level} of $T_{\x}$.

Let $\at T_{\x}$ be the group of all automorphisms of the tree
$T_{\x}$. Let $G<\at T_{\x}$. We recall definition of some
standard subgroups of $G$:
\begin{itemize}
\item
The subgroup of all elements of $G$ fixing every vertex of $n$-th
level, denoted by $\st_G(n)$, is called \emph{the stabilizer of
the $n$-th level}.
\item
For every $v\in \x^*$ the subgroup of all elements of $G$ fixing
every vertex outside the subtree $T_v$, denoted by $\rs_G (v)$, is
called \emph{the rigid stabilizer of the vertex $v$}.
\item
The group generated by the set $\bigcup_{v\in
\x^n}\tp v$,  denoted by $\Rs_G(n)$, is
called \emph{the rigid stabilizer of the $n$th level}.
\end{itemize}

Let $\at_k T_{\x}$ be the subgroup of $\at T_{\x}$ such that an
automorphism $g$ of $T_{\x}$ is in $
\at_k T_{\x} $ if and only if
the equality $g(uv)=g(u)v$ is valid for any $u\in \x^k$ and any
$v\in
\x^*$.
Note that $\at_k T_{\x}$ acts by permutations faithfully on the
set $\x^k$. Note also that $\st_{\at T_{\x}}(k)=\Rs_{\at
T_{\x}}(k)$. Therefore the group $\at T_{\x} $ can be decomposed
into semidirect product of its subgroups $\at_k T_{\x}
\rightthreetimes \Rs_{\at T_{\x}}(p)$. It follows that
 $\at T_{\x} $ is isomorphic to the permutational wreath product $ (\at_k T_{\x},\x^k)
\wr \rs_{\at T_{\x}}(v)$, where $v\in \x^k$ and
$\Rs_{\at T_{\x}}(k)$ is the base subgroup of this wreath product.

We define subgroup $M_0<\at T_{{\x}}$ as infinite direct product:
$M_0=
\prod_{i\ge 0}C_2^{(i)}$, where each $C_2^{(i)}$ is isomorphic to the group of order 2. The action of elements
of $M_0$ on the tree $T_{{\x}}$ can be defined in the following
way. A nontrivial element of $C_2^{(i)}$ acts as follows
$\underbrace{ {0} {0}\ldots
 {0}}_{i} {10}v \arr \underbrace{{0} {0}\ldots
 {0}}_{i} {11}v$, $\underbrace{ {0} {0}\ldots
 {0}}_{i} {11}v \arr \underbrace{{0} {0}\ldots
 {0}}_{i} {10}v$ for every $v\in \x^{(i+1)}$, and $w\arr w$ for
 the other words of $\x^*$.

\begin{lemma}
\label{l00001}
Intersection $M_0\cap (\at T_{{\x}})^2$ is trivial.
\end{lemma}
\begin{proof}
For every $g\in \at T_{\x}$ and $n\ge 0$ we can write
$g=g_n(g_{v_1},\ldots,g_{v_k})$, where $g_n\in \at_n T_{\x}$,
$(g_{v_1},\ldots,g_{v_k})\in \Rs_{\at T_{\x}}(n)$, and
$\{v_1,\ldots, v_k\}=\x^n$. Write $\Pi_n g=\prod_{v\in\x^n}g_v$
for every $n\ge 0$.

It is evident that $\Pi_n g^2$ is an even permutation of
$\x_{n+1}$ for every $g\in \at T_{{\x}}$ and $n\ge 0$. Therefore
$\Pi_n h$ is an even permutation of $\x_{n+1}$ for every $h\in
(\at T_{{\x}})^2$ and $n\ge 0$. But for every nontrivial element
$g\in M_0$ there is $m\ge 0$ such that $\Pi_m g$ is an odd
permutation of $\x_{m+1}$. Thus the intersection $M_0\cap (\at
T_{{\x}})^2$ is trivial.
\end{proof}

\begin{prop}
\label{tm_2011_121}
Let $G$ be an infinite automorphism group of $T_{\x}$ and there is
a positive integer $p$ with the following properties:
\begin{enumerate}
\item
The group $G$ can be decomposed into a semidirect product of its
subgroups $(G\cap
\at_k T_{\x},\x^k)\rightthreetimes
\Rs_G(k)$ provided the group $(G\cap \at_k T_{\x},\x^k)$
satisfies {\bf PS}.
\item
$|M_0\cap G|=|G|$.
\end{enumerate}
Then the group $G$ satisfies the L-condition.
\end{prop}
\begin{proof}
Let $G$ be a group and $k\in
\mathbb{N}$ be such that all conditions of the statement are satisfied. Let $v=0\ldots 0\in \x^k$. Then $G$ is
isomorphic to the permutational wreath product $(G\cap
\at_k T_{\x},\x^k)\wr
\rs_G(v)$.

Consider the subgroup $M=M_0\cap \rs_G(v)$ of the group $G$. It is
obvious that $M$ has exponent 2. Since $M_0\cap G=(M_0 \cap G\cap
\at_k T_{\x}) \times M$ we have $|M|=|M_0\cap G|$.
Combining it with the second condition of the statement we obtain
$|M|=|G|$. It follows that $|M|=|G|=|\rs_G(v)|$. By
Lemma~\ref{l00001} we have $M\cap(\rs_G(v))^2< M_0\cap (\at
T_{\x})^2=E$. Thus the group $G$ satisfies the L-condition by
Proposition~\ref{tm_2011_112}, and the statement follows.
\end{proof}

\subsubsection*{Examples of uncountable groups of automorphisms with the L-condition}
We remind definitions for some classes of automorphisms of
$T_{\x}$.

\begin{itemize}
\item
An automorphism $g$ is called \emph{finitary} if there is a
positive integer $n$ such that the equality $g(uv)=g(u)v$ is valid
for any $u\in \x^n$ and any $v\in
\x^*$.
\item
An automorphism $g$ is called \emph{weakly finitary} if for any
$w\in
\x^\omega$ there are $n\in \mathbb{N}$, $u\in \x^n$, and $v\in
\x^\omega$ such that $w=uv$ and $g(uv)=g(u)v$.
\item
Two words $w_1,w_2\in\x^\omega$ are called \emph{cofinal} if there
are $n\in \mathbb{N}$, $u_1,u_2\in \x^n$, $v\in
\x^\omega$ satisfying $w_1=u_1v$ and $w_2=u_2v$.

An automorphism $g$ is called \emph{cofinal} if it maps cofinal
words to cofinal words.
\item
An automorphism $g$ is called \emph{bicofinal} if both $g$ and
$g^{-1}$ are cofinal.
\end{itemize}

Denote by $\at_f T_{\x}$, $\at_{wf} T_{\x}$, $\at_{b} T_{\x}$ the
sets of all finitary, weakly finitary and bicofinal automorphisms
of $T_{\x}$ respectively. All of these sets are groups.

Note that, by definitions, we have the following inclusions:
$$\at_f T_{\x}<\at_{wf} T_{\x}<\at_{b} T_{\x}.$$

For more details on these groups we refer the reader
to~\cite{sn_c}.

\begin{theorem}
\label{cor_2011_121}
The groups $\at T_{\x}$, $\at_{wf} T_{\x}$ and $\at_{b} T_{\x}$
satisfy the L-condition and so have minimal generating sets.
\end{theorem}
\begin{proof}
Let $G$ be one of the above groups. Then $G$ can be decomposed
into semidirect product of its subgroups $(\at_k
T_{\x},\x^k)\rightthreetimes
\Rs_G(k)$ for a positive integer $k$.
The permutation group $(\at_k T_{\x},\x^k)$ satisfies the
condition {\bf PS} for $k\ge 3$ by Proposition~\ref{pr_1}. It is
clear that $M_0<\at_{wf} T_{\x}<\at_{b} T_{\x}<\at T_{\x}$. It
follows that $G$ satisfies all conditions of
Proposition~\ref{tm_2011_121}, and the statement follows.
\end{proof}

\subsubsection*{Examples of countable groups of automorphisms with lifting condition}

\begin{prop}
\label{cor_countable}
Let $G$ be a countable automorphism group of rooted tree $T_{\x}$
with the following properties:
\begin{itemize}
\item
$\at_f T_{\x}<G$.
\item
$\Rs_G(k)=\st_G(k)$ for some integer $k\ge 3$.
\end{itemize}
Then the group $G$ satisfies the L-condition.
\end{prop}
\begin{proof}
By the condition of the proposition $M_0\cap G > M_0\cap
\at_f T_{\x}$. Since intersection $M_0\cap \at_f T_{\x}$ is countable, $|M_0\cap G | = |G|$.
Since $\at_f T_{\x}<G$, $\at_k T_{\x}<G$. Therefore $G$ is
decomposed into semidirect product $(\at_k
T_{\x},\x^k)\rightthreetimes
\Rs_G(k)$ of its subgroups. The permutation group $(\at_k T_{\x},\x^k)$
satisfies condition {\bf PS} for $k\ge 3$ by
Proposition~\ref{pr_1}. It follows that $G$ satisfies all
conditions of Proposition~\ref{tm_2011_121}, and the statement
follows.
\end{proof}

From now we assume that $X_1=X_2=\ldots$. In this case $T_{\x}$ is
called \emph{regular rooted tree}.

A vertex subtree $T_v$ of $T_{\x}$ for every $v\in
\x^n$ can be naturally identified with the whole tree $T_{\x}$:
$$\pi_v:v x_{n+1} \ldots x_m\mapsto
x_{n+1} x_{n+2} \ldots x_m.$$

Thus for any $g\in \at T_{\x}$ and $v\in \x^*$ we can define
automorphism $g|_v\in
\at T_{\x}$ in the following way: $g|_v(u)=w$ if and only if
$g(vu)=g(v)w$ for every $u,w\in
\x^*$.

An automorphism $g\in \at T_{\x}$ is {\em finite-state}
automorphism if the set of its states is finite. All finite-state
automorphisms of the tree $T_{\x}$ form the \emph{group $\fat
T_{\x}$ of finite-state automorphisms} of the tree $T_{\x}$.

Let us define the number $\Theta_n(g)=\#\{v\in
\x^n
\ |\ g|_v
\ne e\} $ for every $g\in \at T_{\x}$.

The set of all finite-state automorphisms $g\in \fat T_{\x}$ such
that the sequence $\Theta_n(g)$ is bounded by a polynomial of
degree $m$ form the \emph{group $\pol(m)$ of polynomial
automorphisms of degree} $m$ of the tree $T_{\x}$. The group
$\pol(0)$ is also called the \emph{group of bounded
automorphisms}. The group of \emph{polynomial automorphisms},
denoted by $\pol(\infty)$, is defined to be the union of
increasing chain of groups: $\pol(\infty)=\bigcup_{m=0}^{\infty}
\pol(m)$.

A subgroup $G$ of $\at T_{\x}$ is \emph{self-similar} provided
$g|_v \in G$ for all $g\in G$ and $v\in\x^*$. \emph{The group
$\rat T_{\x}$ of functionally recursive automorphisms} of $T_{\x}$
 can be defined as the union of all finitely
generated self-similar subgroups of $\at T_{\x}$.

We refer the reader to~\cite{sid:cycl, brsi:basegr,sidki_monogr}
for details concerning groups defined above.

\begin{theorem}
\label{cor_2011_121_2}
The groups $\fat T_{\x}$, $\pol(m)$ ($m\ge 0$), $\pol(\infty)$,
and $\rat T_{\x}$ satisfy the L-condition and so have minimal
generating sets.
\end{theorem}
\begin{proof}
Let $G$ be one of the above groups. It is well known that G is
countable group. By definition, the group $G$ contains the group
$\at_f T_{\x}$. Furthermore, we have $\st_G(k)=\Rs_G(k)$ for all
positive integer $k$. It follows that $G$ satisfies the
L-condition by Proposition~\ref{cor_countable}, and the theorem
follows.
\end{proof}

\bibliographystyle{amsplain}

\bibliography{my2,mymath_d,lavrenyuk,nekrash_d}

\end{document}